\documentclass[12pt]{amsart}
\usepackage{amsfonts}
\usepackage{amsmath}
\usepackage{amssymb}
\usepackage[margin=1.4in]{geometry}
\usepackage{verbatim}
\newtheorem{theorem}{Theorem}[section]
\newtheorem{corollary}[theorem]{Corollary}
\newtheorem{lemma}[theorem]{Lemma}

\theoremstyle{definition}

\newtheorem{example}[theorem]{Example}

\renewcommand{\a}{{\alpha}}
\renewcommand{\b}{{\beta}}

\newcommand{\g}{{\gamma}}

\renewcommand{\ss}{{\sigma_o}}

\newcommand{\cP}{\mathcal{P}}
\newcommand{\cN}{\mathcal{N}}

\numberwithin{equation}{section}

\begin{document}
\title{$M(x)=o(x)$ Estimates for Beurling numbers}

\author[G.~Debruyne]{Gregory Debruyne}
\thanks{G.~Debruyne gratefully acknowledges support by Ghent
  University, through a BOF Ph.D. grant} 
\address{G.~Debruyne\\ Department of Mathematics\\ Ghent University\\
  Krijgslaan 281\\ B 9000 Ghent\\ Belgium} 
\email{gdbruyne@cage.UGent.be}

\author[H.~G. Diamond]{Harold G. Diamond}
\address{H.~G.~Diamond\\ Department of Mathematics\\ University of
  Illinois\\ 1409 W. Green St.\\ Urbana IL 61801\\ U.S.A.} 
\email{diamond@math.uiuc.edu}

\author[J.~Vindas]{Jasson Vindas} 
\thanks{The work of J. Vindas was supported by the Research
  Foundation--Flanders, through the FWO-grant number 1520515N} 
\address{J.~Vindas\\ Department of Mathematics\\ Ghent University\\
  Krijgslaan 281\\ B 9000 Ghent\\ Belgium} 
\email{jvindas@cage.UGent.be}


\begin{abstract}
In classical prime number theory there are several asymptotic
formulas said to be ``equivalent'' to the PNT.  One is the bound $M(x)
= o(x)$ for the sum function of the Moebius function.  For Beurling
generalized numbers, this estimate is not an unconditional consequence
of the PNT. Here we give two conditions that yield the Beurling
version of the $M(x)$ bound, and examples illustrating failures
when these conditions are not satisfied.

%\bigskip

%\noindent \textsc{R\'{e}sum\'{e}.} Dans la th\'eorie des nombres premiers classiques, certaines
%expressions sont consider\'ees d'\^etre ``\'equivalent'' au TNP
%(Th\'eor\`eme des nombres premiers).  Parmi elles est la borne $M(x) =
%o(x)$ pour la fonction sommatoire de Moebius. En cas des nombres
%premiers, g\'en\'eralis\'es par Beurling, cette borne ne suit pas
%n\'ecessairement du TNP sans exiger des hypoth\`eses
%additionelles. Ici, deux conditions sont pr\'esent\'ees, impliquant la
%version Beurling pour la borne sur $M(x)$ et quelques examples sont
%construits, d\'emontrant l'absence \'eventuelle de cette borne si ces
%conditions ne sont pas r\'ealis\'ees.
\end{abstract}

\maketitle

\section{Introduction}  \label{sec:intro}

Let $\mu(n)$ denote the Moebius arithmetic function and $M(x)$ its sum
function. Von Mangoldt first established the estimate $M(x) = o(x)$,
essentially going through the steps used in proving the Prime Number
Theorem (PNT).  A few years later, Landau showed by relatively simple
real variable arguments that this and several other estimates followed
from the PNT \cite[\S150]{landau}.  Similarly, these relations imply each
other and the PNT; thus they are said to be ``equivalent'' to the PNT.

In this article we consider an analog of the $M$ bound for Beurling
generalized (g-)numbers.  A Beurling g-prime system is an unbounded
sequence of real numbers ${\cP} =\{p_i\}$ satisfying $1 < p_1 \le p_2
\le \ldots\,$, and the multiplicative semigroup generated by $\cP$ and
1 is called the associated collection $\cN$ of g-integers
(cf.~\cite{BD69}, \cite{diamond-zhangbook}, \cite{MV07}).  The
counting function $N(x)$ of $\cN$ is the number of elements of $\cN$
not exceeding $x$.  The g-Chebyshev function of $\cN$ is
\[
\psi(x) = \sum_{p_i^{\alpha_i} \le x} \log p_i.
\]
(Here and below we give our Beurling functions the same names used in
classical number theory.)

The PNT-related assertions are somewhat different  
for g-numbers: not all implications between the several corresponding
assertions hold unconditionally, see e.g.~\cite{d-vPNTequiv2016},
\cite[Chap.~14]{diamond-zhangbook}, \cite{zhang1987}.  Here we study
the Beurling version of the assertion $M(x) = o(x)$ and show this can
be deduced (a) from the PNT under an O-boundedness condition
$N(x)=O(x)$ or (b) from a Chebyshev-type upper bound assuming
$N(x)\sim ax$ (for $a>0$) and the integral condition \eqref{L1cond}
(below).  At the end, we give examples in which $M$ estimates fail.

A word about the Beurling version of $M(x)$.  The characteristic
property of the Moebius function is its being the (multiplicative)
convolution inverse of the $1$ function.  For g-numbers, we define the
measure $\mathrm{d}M$ as the convolution inverse of $\mathrm{d}N$; by familiar Mellin
transform properties (cf. \cite{BD04}, \cite{diamond-zhangbook},
\cite{MV07}), 
\[
 \int_{1-}^\infty u^{-s} \mathrm{d}M(u) = 1/\zeta(s) \ \, {\rm with} \ \,
\zeta(s) = \zeta_{\cN}(s) = \int_{1-}^\infty u^{-s} \mathrm{d}N(u),
\]
the Beurling number version of the Riemann zeta function. Note that
these definitions make sense even when $\mathrm{d}N$ is not discrete or when
factorization into primes is not unique.
\smallskip

We thank Wen-Bin Zhang for his useful comments.

\section{PNT hypothesis}  \label{sec:PNT}

In this section we show that the PNT together with the ``O-density''
condition $N(x) \ll x$ implies $M(x) = o(x)$.
\begin{theorem}  \label{th:M}
Let $\cN$ be a g-number system for which the PNT holds and $N(x) \ll
x$.   Then $M(x) = o(x)$. 
\end{theorem}

Other sufficient conditions for an $M$ estimate are known, 
e.g.~\cite[ Prop.~14.10]{diamond-zhangbook}, if the PNT holds for a
g-number system and the integer counting function satisfies the
logarithmic density condition
\begin{equation}
\label{log-density}
\int_{1^{-}}^{x}\frac{\mathrm{d}N(t)}{t}\sim a\log x,
\end{equation}
then $M(x) = o(x)$.  The present result differs from the other in that
log-density and O-density are conditions that do not imply one another;
also, the proofs are very different.

The key to our argument is the following relation.
\begin{lemma} \label{lem:key}
Under the hypotheses of the theorem,
\begin{equation}  \label{eq:main-ident}
\frac{M(x)}{x} = \frac {-1}{\log x}\int_1^x \!\frac{M(t)}{t^2}\,
\mathrm{d}t + o(1). 
\end{equation}
\end{lemma}

\begin{proof}[Proof of the lemma]
A variant of Chebyshev's identity for primes reads
\[
L\mathrm{d}M =  -\mathrm{d}M * \mathrm{d}\psi,
\]
where $L$ is the operator of multiplication by $\log t$ and $*$ is
multiplicative convolution.  This can be verified (in the classical or
in the Beurling case) via a Mellin transform.  Note that this
transform carries convolutions into pointwise products and the $L$
operator into differentiation.  The equivalent Mellin formula is the
identity
\[
\Big\{\frac{1}{\zeta(s)} \Big\}' = -\frac{1}{\zeta(s)} \cdot
\frac {\zeta'(s)}{\zeta(s)}.
\]

Now add and subtract the term $(\delta_1+\mathrm{d}t)*\mathrm{d}M$ in
the variant of the
Chebyshev relation, with $\delta_1$ the point mass at 1 and $\mathrm{d}t$
the Lebesgue measure on $(1,\, \infty)$.  Integrating, we find
\begin{equation}  \label{eq:Cheb'}
\int_{1-}^x L\,\mathrm{d}M =  \int_{1-}^x\mathrm{d}M *
(\delta_1+\mathrm{d}t - \mathrm{d}\psi) - \int_{1-}^x\mathrm{d}M 
* (\delta_1+\mathrm{d}t).
\end{equation}

Integrating by parts the left side of the last formula, we get
\[
M(x) \log x - \int_1^x \frac{M(t)}{t}\, \mathrm{d}t = M(x) \log x + O(x),
\]
since $|\mathrm{d}M| \le \mathrm{d}N$, and hence $|M(t)| \le N(t) \ll
t$.  If we evaluate the convolution integrals by the iterated integral
formula
\[
 \int_{1-}^x\mathrm{d}A * \mathrm{d}B = \int\!\!\!\int_{st \le x}\!
 \mathrm{d}A(s)\, \mathrm{d}B(t) = \int_{1-}^xA(x/t)\, \mathrm{d}B(t)
\]
and use the PNT, we find the first term on the right side of
\eqref{eq:Cheb'} to be 
\[
\int_{1-}^x \{x/t-\psi(x/t)\}\,\mathrm{d}M(t) = 
\int_{1-}^x o(x/t)\,\mathrm{d}N(t) = o(x \log x).
\]
Also, the last term of \eqref{eq:Cheb'} is, upon integrating by parts,
\[
\int_{1-}^x \frac xt \, \mathrm{d}M(t) = M(x) + x\!\int_1^x
\!\frac{M(t)}{t^2}\,\mathrm{d}t. 
\]

Finally, we combine the bounds for the terms of \eqref{eq:Cheb'},
divide through by $x \log x$, and note that $M(x) \ll x$ to obtain
\eqref{eq:main-ident}.
\end{proof}

\begin{proof}[Proof of Theorem \ref{th:M}]
First, we can assume that $M(x)$ has an infinite number of sign
changes.  Otherwise, there is some number $z$ such that $M(x)$ is of
one sign for all $x \ge z$.  By \eqref{eq:main-ident}, as $x \to \infty$,
\[
\frac{M(x)}{x}
= \frac {-1}{\log x}\, \int_1^z
\!\frac{O(t)}{t^2}\,\mathrm{d}t -\frac{1}{\log x}\int_z^x
\!\frac{M(t)}{t^2}\,\mathrm{d}t  + o(1).
\]
Thus we have
\[
\frac{M(x)}{x} 
+ \frac{1}{\log x}\int_z^x \!\frac{M(t)}{t^2}\,\mathrm{d}t
=  \frac {O(\log z)}{\log x} +o(1) = o(1).
\]
Since $M(x)$ and the integral are of the same sign, $M(x)/x \to 0$ 
as $x \to \infty$, and this case is done.

Now suppose that $M(x)$, which we regard as a right-continuous
function, changes sign at $x$.  We show that $M(x)/x = o(1)$.  If
$M(x)/x = 0$, there is nothing more to say here.  If, on the other
hand, $M(x) > 0$, then there is a number $y \in (x-1,\ x)$ with $M(y)
\le 0$.  If we apply \eqref{eq:main-ident} again, we find
\begin{align*}
\frac{M(x)}{x}
&= \frac {-\log y}{\log x}\, \frac{1}{\log y}\int_1^y
\!\frac{M(t)}{t^2}\,\mathrm{d}t -\frac{1}{\log x}\int_y^x
\!\frac{M(t)}{t^2}\,\mathrm{d}t  + o(1) \\
&= \frac {\log y}{\log x}\,\frac{M(y)}{y} + 
\frac{1}{\log x} \int_y^x \frac{O(t)\,\mathrm{d}t}{t^2} +o(1). 
\end{align*}
Thus $M(x)/x - M(y)/y  = o(1)$, and since $M(x)/x> 0 \ge M(y)/y$,
each is $o(1)$.  A similar story holds if $M(x) < 0$.

Finally, suppose that $M(t)$ changes sign at $t=y$ (so that
$M(y)/y = o(1)$) and  $M(t)$ is of one sign 
for $y < t \le z$.  By yet another application of
\eqref{eq:main-ident},  we find for any $x \in (y,\, z]$
\[
\frac{M(x)}{x}
= \frac {-\log y}{\log x}\, \frac{1}{\log y}\int_1^y
\!\frac{M(t)}{t^2}\,\mathrm{d}t -\frac{1}{\log x}\int_y^x
\!\frac{M(t)}{t^2}\,\mathrm{d}t  + o(1)
\]
or
\[
\frac{M(x)}{x} + \frac{1}{\log x}\int_y^x \!\frac{M(t)}{t^2}\,\mathrm{d}t 
= \frac {\log y}{\log x}\,\frac{M(y)}{y} +o(1) = o(1)
\]
as $y \to \infty$.  Since $M(x)/x$ and the integral are of the same
sign, it follows that $M(x)/x = o(1)$.
\end{proof}

\section{Chebyshev hypothesis}  \label{sec:cheb}

What happens if the PNT hypothesis of the last theorem is weakened to
just a Chebyshev upper bound? In the examples of Section \ref{Sect
  examples} we show that even two-sided Chebyshev estimates by
themselves are not strong enough to ensure that $M(x) = o(x)$ holds.
Furthermore, this bound could fail even if, in addition to Chebyshev
estimates, one also assumes that $N$ satisfies both
(\ref{log-density}) and $N(x)\ll x$, see Example \ref{ex1C-M}. We
shall show, however, that if the regularity hypothesis on $N$ is
slightly augmented, then one can indeed deduce the desired $M$ bound.

\begin{theorem} \label{thCheby-moebius} Suppose that a g-number
 system satisfies the following conditions: \begin{enumerate}
 \item [(a)] a Chebyshev upper bound, that is,  
\begin{equation}
\label{chebyeq1} \limsup_{x\to\infty} \frac{\psi(x)}{x} < \infty,
\end{equation}
\item [(b)]  for some positive constant $a, \ N(x) \sim ax$,
\item [(c)] for some $\b \in (0,1/2)$ and all $\sigma \in (1,2)$
\begin{equation}
\label{L1cond} \int_{1}^{\infty} \frac{|N(x)-ax|}{x^{\sigma+1}}\:
\mathrm{d}x \ll (\sigma-1)^{-\b}\,.
\end{equation} 
 \end{enumerate}
Then, $M(x)=o(x)$ holds.
\end{theorem}

As a simple consequence we have an improvement of a result of
W.-B.~Zhang \cite[Cor.~2.5]{zhang1987}:
\begin{corollary}  If a g-number system satisfies \eqref{chebyeq1} and
\[
N(x) - ax \ll x \log ^{-\gamma} x
\]
for some $\gamma > 1/2$, then $M(x)=o(x)$ holds.
\end{corollary}

 W.-B. Zhang had conjectured\footnote{Oral communication to the authors.} that a Chebyshev
 bound along with the $L^{1}$ bound 
\begin{equation}
\label{strongL1cond} \int_{1}^{\infty} \frac{|N(x)-ax|}{x^{2}}
\mathrm{d}x <\infty
\end{equation} 
imply $M(x) = o(x)$.  Naturally, (\ref{strongL1cond}) is included in
(\ref{L1cond}) and $N(x)\sim ax$. 

It is easy to show that $M(x) = o(x)$ is always implied by
\begin{equation} \label{eqlandausharprelation}
 m(x) =\int_{1^{-}}^{x}\frac{\mathrm{d}M(u)}{u}=
 % \sum_{n_{k}\leq x} \frac{\mu(n_{k})}{n_{k}} 
 o(1),
\end{equation}
but the converse implication is not true in general
\cite{diamond-zhangbook}.  On the other hand, it has recently been
shown \cite[Thm. 2]{d-vPNTequiv2016} that the assertions become
equivalent under the additional hypothesis (\ref{strongL1cond}).  So, Theorem
\ref{thCheby-moebius} and the quoted result yield at once:

\begin{corollary} \label{corCheby-moebius} Under (\ref{chebyeq1}) and
  (\ref{strongL1cond}), relation (\ref{eqlandausharprelation})
  holds as well. 
\end{corollary}

We can also strengthen another result of W.-B. Zhang
\cite[Thm. 2.3]{zhang1987}.  

\begin{corollary} \label{cor2Cheby-moebius} The condition
(\ref{strongL1cond})
and 
\begin{equation}\label{avcondpointwiseeq}
\int_{1}^{x}\frac{(N(t)-at)\log t}{t}\:\mathrm{d}t\ll x
\end{equation}
imply (\ref{eqlandausharprelation}).
\end{corollary}
\begin{proof} Chebyshev bounds are known to hold under the hypotheses
  (\ref{strongL1cond}) and (\ref{avcondpointwiseeq})
  \cite[Thm. 11.1]{diamond-zhangbook}. The rest follows from Corollary
  \ref{corCheby-moebius}. 
\end{proof}

We shall prove Theorem \ref{thCheby-moebius} using several lemmas.
Our method is inspired by W.-B. Zhang's proof of a H{a}l\'{a}sz-type
theorem for Beurling primes \cite{zhang1987}.  Our first step is to
replace $M(x)=o(x)$ by an equivalent asymptotic relation.

\begin{lemma}\label{lem1C-M} Suppose that the integer counting
  function of a g-number system has a positive density, i.e.,
  $N(x)\sim a x$ for some $a>0$. Then, $M(x)=o(x)$ if and only if
\begin{equation}
\label{MoebiusAv} f(x)=\int_{1}^{x}\left(\int_{1}^{u}\log t\:
  \mathrm{d}M(t) \right)\frac{\mathrm{du}}{u}=o(x\log x). 
\end{equation}
\end{lemma}
\begin{proof}The direct implication is trivial. For the converse, we
  set 
\[
g(x)=\int_{1^{-}}^{x} \log t\, (\mathrm{d}N(t)+\mathrm{d}M(t)). 
\]
Notice that 
\[
\mathrm{d}N+\mathrm{d}M=2\sum_{n=0}^{\infty}\,
\mathrm{d}\Pi^{\ast 2n}/(2n)!
\]
is a non-negative measure, so that $g$ is non-decreasing. The
hypotheses (\ref{MoebiusAv}) and $N(x)\sim ax$ give
$$
\int_{1}^{x}\frac{g(u)}{u}\,\mathrm{d}u\sim  a x\log x.
$$
Hence,
$$
\int_{1}^{x}g(u)\,\mathrm{d}u= x \int_{1}^{x}\frac{g(u)}{u} \,\mathrm{d}u
- \int_{1}^{x}\!\left(\int_{1}^{t}\frac{g(u)}{u} \,\mathrm{d}u\right)
\mathrm{d}t \sim \frac{a}{2} x^{2}\log x.  
$$
Since $g$ is a non-decreasing function, we see by a simple
differencing argument (see e.g. \cite[p. 34]{korevaarbook}) that
$g(x)\sim a x\log x$. But also we have $\int_{1^{-}}^{x}\log u\,
\mathrm{d}N(u)\sim a x\log x $; consequently, $\int_{1^{-}}^{x}\log
u\, \mathrm{d}M(u) = o( x\log x)$.  Integration by parts now yields
$$
\log x\, M(x)= \int_{1}^{x}\log u\, \mathrm{d}M(u) 
+\int_{1}^{x} \frac{M(u)}{u}\, \mathrm{d}u = o(x\log x)+O(x),
$$
and the result then follows by dividing by $\log x$.
\end{proof}

The next lemma provides a crucial analytic estimate.
\begin{lemma}\label{lem3CC-M} Suppose that $N(x)\sim a x$, with $a>0$. Then, 
\begin{equation}
\label{invzeta}
\frac{1}{\zeta(\sigma+it)}=o\left(\frac{1}{\sigma-1}\right), \quad
\sigma \to 1^{+},
\end{equation}
uniformly for $t$ on compact intervals.
\end{lemma}
\begin{proof} We first show that (\ref{invzeta}) holds pointwise,
i.e., for each fixed $t\in \mathbb{R}$, without the uniformity
requirement. If $t=0$ this is clear because $\zeta(\sigma)\sim
a/(\sigma-1)$ and thus $1/\zeta(\sigma)=o(1)$. Note that $N(x)\sim a x$ 
implies  
$$
\zeta(s)=\frac{a}{s-1}+o\left(\frac{|s|}{\sigma-1}\right)
$$
uniformly. If $t\neq 0$, we obtain, $\zeta(\sigma+2it)=o_{t}(1/(\sigma-1))$. 
Applying the 3-4-1 inequality, we conclude that 
$$
1\leq |\zeta(\sigma)|^{3}|\zeta(\sigma+it)|^{4} |\zeta(\sigma+2it)|
=|\zeta(\sigma+it)|^{4} \,o_{t}\!\left(\frac{1}{(\sigma-1)^{4}}\right), 
$$
which shows our claim. Equivalently, we have 
$$
\exp\left(-\int_{1}^{\infty}x^{-\sigma}(1+\cos(t\log x))
\,\mathrm{d}\Pi(x)\right)
=\frac{1}{|\zeta(\sigma)\zeta(\sigma+it)|}=o_{t}(1).
$$
The left-hand side of the last formula is a net of continuous
monotone functions in the variable $\sigma$ that tend pointwise to 0 as $\sigma\to 1^{+}$;
Dini's theorem then asserts that it must also converge to 0 uniformly
for $t$ on compact sets, as required. 
\end{proof}

The next lemma is very simple but useful.

\begin{lemma}\label{lem3C-M}Let $F$ be a right continuous function of
  local bounded variation with support in $[1,\infty)$ satisfying the
  bound $F(x)=O(x)$. Set $\widehat{F}(s)=\int_{1^{-}}^{\infty}x^{-s}
  \mathrm{d}F(x)$, $\Re s>1$. Then,  
$$
\int_{\Re s=\sigma} \left| \frac{\widehat{F}(s)}{s}\right|^{2}
|\mathrm{d}s|\ll \frac{1}{\sigma-1} \,.
$$
\end{lemma}
\begin{proof}
Indeed, by the Plancherel theorem,
$$
\int_{\Re s=\sigma}  \left| \frac{\widehat{F}(s)}{s}\right|^{2}
|\mathrm{d}s|=2\pi\! \int_{0}^{\infty}\!\! e^{-2\sigma x}
|F(e^{x})|^{2}\,\mathrm{d}x\ll \int_{0}^{\infty}\!\! e^{-2(\sigma-1) x}
\,\mathrm{d}x= \frac{1}{2(\sigma-1)}\,. \qedhere
$$
\end{proof}

If a g-number system satisfies a Chebyshev upper bound, then Lemma
\ref{lem3C-M} implies  
\begin{equation}
\label{L2Chevy}
\int_{\Re s=\sigma} \left|\frac{\zeta'(s)}{s\zeta(s)}\right|^{2} 
|\mathrm{d}s|\ll \frac{1}{\sigma-1}.
\end{equation}
Similarly, $N(x)=O(x)$ yields
\begin{equation}
\label{L2zeta}
\int_{\Re s=\sigma} \left|\frac{\zeta(s)}{s}\right|^{2} |\mathrm{d}s|
\ll \frac{1}{\sigma-1}.
\end{equation}

Also, we shall need the following version of the Wiener-Wintner theorem
\cite{bateman-diamond2000,montgomeryBook1970}.  
\begin{lemma}\label{lem4C-M} Let $F_{1}$ and $F_{2}$ be right
continuous functions of local bounded variation with support in
$[1,\infty)$. Suppose that their Mellin-Stieltjes transforms
$\widehat{F}_{j}(s)=\int_{1^{-}}^{\infty}x^{-s} \mathrm{d}F_{j}(x)$
are convergent on $\Re s>\alpha$, that $F_{1}$ is non-decreasing, and
$|\mathrm{d}F_{2}|\leq \mathrm{d}F_{1}$. Then, for all
$b\in\mathbb{R}$, $c>0$, and $\sigma>\alpha$,
$$
\int_{b}^{b+c}|\widehat{F}_{2}(\sigma+it)|^{2}\,\mathrm{d}t
\leq 2 \int_{-c}^{c}|\widehat{F}_{1}(\sigma+it)|^{2}\,\mathrm{d}t.
$$
\end{lemma}

We are ready to present the proof of Theorem \ref{thCheby-moebius}.
\begin{proof}[Proof of Theorem \ref{thCheby-moebius}] 
In view of Lemma
  \ref{lem1C-M}, it suffices to show (\ref{MoebiusAv}).  
The Mellin-Stieltjes transform of the function $f$ is 
$$
-\frac{1}{s}\left(\frac{1}{\zeta(s)}\right)'.
$$
Given $x > 1$, it is convenient to set $\ss=1+1/\log x$.  By the
Perron inversion formula, the Cauchy-Schwarz inequality, and
(\ref{L2Chevy}), we have
\begin{align*}
\nonumber
\frac{|f(x)|}{x}&=\frac{1}{2\pi}\left|\int_{\Re s=\ss}  
\frac{x^{s-1}\zeta'(s)}{s^{2}\zeta^{2}(s)}\, \mathrm{d}s
\right|\leq \frac{e}{2\pi} \int_{\Re s=\ss}
\left|\frac{\zeta'(s)}{s^{2}\zeta^{2}(s)}\right| |\mathrm{d}s|  \\
&
\ll \log^{1/2}x \left( \int_{\Re s=\ss} \left|
\frac{1}{s\zeta(s)}\right|^{2} |\mathrm{d}s| \right)^{1/2}.
\end{align*}

Next, we take a large number $\lambda$, arbitrary but fixed. We
split the integration line $\{\Re s=\ss\}$ of the last integral into two
parts, $\{\ss+it:\: |t|\geq \lambda\} $ and $\{\ss+it:\: |t|\leq
\lambda\} $, and we denote the corresponding integrals over these sets
as $I_{1}(x)$ and $I_{2}(x)$ respectively, so that
\begin{equation}
\label{eq1proofM}
\frac{|f(x)|}{x}\ll \big((I_{1}(x))^{1/2}+(I_{2}(x))^{1/2}\big) \log^{1/2}x.
\end{equation}

To estimate $I_{1}(x)$, we apply Lemma \ref{lem4C-M} to
$|\mathrm{d}M|\leq \mathrm{d}N$ and employ  (\ref{L2zeta}), 
$$
\left(\int_{-\lambda -m-1}^{-\lambda-m} +\int_{\lambda+m}^{\lambda+m+1}
\right)\left|\frac{1}{\zeta(\ss+it)}\right|^{2} \mathrm{d}t\leq 4
\int_{-1}^{1} \left|\zeta(\ss+it)\right|^{2} \mathrm{d}t \ll \log
x. 
$$
Therefore, we have the bound
\begin{equation}
\label{eq2proofM}
I_{1}(x)\ll \log x \sum_{m=0}^{\infty} \frac{1}{1+m^{2}
+\lambda^{2}}\ll \frac{\log x}{\lambda}.
\end{equation}

To deal with $I_{2}(x)$, we need to derive further properties of the
zeta function. Using the hypothesis \eqref{L1cond}, we find that
\begin{align*}
\zeta(s)-\frac{a}{s-1}&= s \!\int_{1}^{\infty}\!\!
x^{-s}\frac{N(x)-ax}{x}\:\mathrm{d}x +a \\
&\ll |s|\!\int_{1}^{\infty}\!\! x^{-\sigma}\frac{|N(x)-ax|}{x}\,
\mathrm{d}x =O_t\!\left(\frac{1}{(\sigma-1)^{\beta}}\right). 
\end{align*}
Hence, we obtain
\begin{equation}
\label{eq3proofM}
\zeta(\ss+it)= \frac{a}{\ss-1+it}
+ O(\log^{\beta} x)
\end{equation}
for some number $\beta \in (0,\,1/2)$, uniformly for $t$ on
compact sets. We are ready to estimate $I_{2}(x)$. Set
$\eta=(1-2\beta)/(1-\beta)$ and note that $\eta \in (0,\, 1)$.  Then,
using Lemma~\ref{lem3CC-M},
$$
I_{2}(x)\leq \int_{-\lambda}^{\lambda} \left|
\frac{1}{\zeta(\ss+it)}\right|^{2}\mathrm{d}t\leq 
\left(\int_{-\lambda}^{\lambda} \left|
\frac{1}{\zeta(\ss+it)}\right|^{2-\eta}\mathrm{d}t\right) 
 o_{\lambda}(\log^{\eta}x).
$$   
On the other hand, applying  Lemma \ref{lem4C-M} to 
$\mathrm{d}F_{1}=\exp^{\ast}(-(1-\eta/2)\,\mathrm{d}\Pi)$ and
$\mathrm{d}F_{2}=\exp^{\ast}((1-\eta/2)\,\mathrm{d}\Pi)$
and using  (\ref{eq3proofM}), we find 
\begin{align*}
\int_{-\lambda}^{\lambda}\left|\frac{1}{\zeta(\ss+it)}\right|^{2-\eta}\mathrm{d}t
&
\leq 4\int_{-\lambda}^{\lambda} \left|\zeta(\ss+it)\right|^{2-\eta}\mathrm{d}t
\\
&
 \ll  \int_{-\lambda}^{\lambda} \frac{\mathrm{d}t}{ ((\ss-1)^{2}+t^{2})^{1-\eta/2}}
+\lambda\log^{(2-\eta)\beta}x\ll \lambda \log^{1-\eta}x,
\end{align*}
which implies $I_{2}(x)= o_{\lambda}(\log
x)$. Inserting this and the bound (\ref{eq2proofM}) into
(\ref{eq1proofM}), we arrive at 
$
|f(x)|/(x\log x)\ll \lambda^{-1/2} +o_{\lambda}(1).$ Taking first the limit
superior as $x\to\infty$ and then $\lambda\to\infty$, we have shown
that 
$$
\lim _{x\to\infty} \frac{f(x)}{x\log x}=0.
$$
By Lemma~\ref{lem1C-M}, $M(x) = o(x)$.
This completes the proof of Theorem \ref{thCheby-moebius}.
\end{proof}

\section{Three examples}\label{Sect examples}
The examples of this section center on the importance of
$N(x)$ being close to $ax$ in Theorem \ref{thCheby-moebius}.  In
the first example, $N(x)/x$ has excessive wobble and in the second
one, excessive growth; in both cases $M(x) = o(x)$ fails. The third
example shows that condition~\eqref{L1cond} is not sufficient to
insure the convergence of $N(x)/x$, whence the introduction of this
hypothesis. 

In preparation for treating the first two examples, we give a
necessary condition for $M(x) = o(x)$. An analytic function $G(s)$ on
the half-plane $\{s:\:\Re s>\alpha\}$ is said to have a right-hand zero of
order $\beta>0$ at $s=it_{0}+\alpha$ if $\lim_{\sigma\to\alpha^{+}}
(\sigma -\alpha)^{-\beta}G(it_{0}+\sigma)$ exists and is non-zero.
Our examples violate the following necessary condition:
\begin{lemma}
\label{lexC-M}
If $M(x)=o(x)$, then $\zeta(s)$ does not have any right-hand zero of
order $\geq 1$ on $\{\Re s=1\}$. 
\end{lemma}
\begin{proof} We must have
$$
\frac{1}{\zeta(\sigma+it)}=s\!\int_{1}^{\infty} x^{-it-\sigma}
\frac{M(x)}{x}\:\mathrm{d}x= o\!\left(\frac{1}{\sigma-1}\right), \quad
\sigma\to 1^{+},  
$$
uniformly for $t$ on compact intervals.
\end{proof}

\begin{example}
\label{ex1C-M}
We consider
$$
\Pi (x)=\sum_{2^{k+1/2}\leq x} \frac{2^{k+1/2}}{k}\,.
$$
This satisfies the Chebyshev bounds: we have
\[
\psi(x)=  2^{\lfloor \log x/\log 2 + 1/2\rfloor +1/2}\log 2 +O(x/\log x).
\]
Thus
$$
\liminf_{x\to\infty} \frac{\psi(x)}{x}=\log 2 \quad \mbox{and} \quad
\limsup_{x\to\infty} \frac{\psi(x)}{x}=2 \log 2. 
$$
Further, the zeta function of this g-number system can be
explicitly computed:  
$$
\log \zeta(s)= 2^{-(s-1)/2}\sum_{k=1}^{\infty} \frac{2^{-k(s-1)}}{k} = - 2^{-(s-1)/2}\log(1-2^{-(s-1)}).
$$
We conclude that 
$$\zeta(\sigma)\sim \frac{1}{(\sigma-1)\log 2}, \quad \sigma\to1^{+};$$
therefore, by the Hardy-Littlewood-Karamata Tauberian Theorem
\cite{diamond-zhangbook}, \cite{korevaarbook}, $N$ has logarithmic density 
\begin{equation}
\label{log-densityex1}
\int_{1^{-}}^{x} \frac{\mathrm{d}N(u)}{u}\sim \frac{\log x}{\log 2}.
\end{equation}
Furthermore, $\zeta(s)$  has infinitely many right-hand zeros of order
1 at the points $s=1\pm i 2\pi (2n+1) /\log 2$, $n\in\mathbb{N}$, because
\[
\zeta\left(\sigma \pm i\frac{2\pi(2n+1)}{\log 2}\right)= \frac{1}{
  \zeta (\sigma)}\sim (\sigma-1)\log 2. 
\]
It follows, by Lemma \ref{lexC-M}, that
$$M(x)=\Omega(x).$$

To show the wobble of $F(x) = N(x)/x$, we apply an idea of Ingham
\cite{ingham1942} that is based on a finite form of the Wiener-Ikehara
method.  We use (essentially) the result given in \cite[Thm.~11.12]{BD04}.
For $\Re s > 0$,
\[
\int_1^\infty x^{-s-1} F(x)\,\mathrm{d}x = \frac{\zeta(s+1)}{s+1} = G(s).
\]
The discontinuities of $G(s)$ on the line segment $(-8\pi i / \log 2,
8\pi i / \log 2)$, which provide a measure of the wobble, occur at
$s=0$ and $s = \pm 4\pi i / \log 2$.

We analyze the behavior of $G$ near these points.  Let $s$ be a
complex number with $\Re s \ge 0$ (to avoid any logarithmic fuss)
and $0 <|s| \le 1/2$ (so $\log s \ll \log |1/s|$ is valid).
 For $n = 0, \, \pm 1$, a small calculation shows that
\[
\zeta(1+ s + 4 \pi n i/\log 2) = \frac{1}{s \log 2}
+ O(\log |1/s|).
\]
For $n = -1,\,0, \, 1$ set
\[
\gamma_n =  4 \pi n/\log 2, \quad \  \a_n = 1/(\log 2+ 4 \pi in).
\]
Take $T$ a number between $4\pi/\log 2$ and $8\pi/\log 2$, e.g.~$T=36$,
and set
\[
G^*(s) =  \sum_{-1\le n \le 1}\frac{\a_{n}}{s-i \gamma_n}, \quad 
F_T^*(u) = \sum_{-1\le n \le 1}\!\!\a_{n}\Big(1 - \frac{|\gamma_{n}|}{T}
\,\Big)\, e^{i \gamma_nu}.
\]

Now $G - G^*$ has a continuation to the closed strip $\{s\!: \sigma
\ge 0, \, |t| \le T\}$ as a function that is continuous save
logarithmic singularities at $\g_{-1},\,\g_0,\,\g_1$.  In particular,
$G - G^*$ is integrable on the imaginary segment $(-iT, \, iT)$.  The
result in \cite{BD04} is stated for an extension that is continuous at
all points of such an interval, but integrability is a sufficient 
condition for the result to hold.

We find that 
\[
\liminf_{u \to \infty} F(u) \le \inf_u F_T^*(u) < \sup_uF_T^*(u) 
\le \limsup_{u \to \infty} F(u).
\]
By a little algebra,
\begin{align*}
\sup_uF_T^*(u) &= \frac 1{\log 2} + \frac 2{|\log 2 - 4 \pi i|}\Big(1 -
\frac{4 \pi}{36 \log 2}\Big) > 1.52, \\
\inf_uF_T^*(u) &= \frac 1{\log 2} - \frac 2{|\log 2 - 4 \pi i|}\Big(1 -
\frac{4 \pi}{36 \log 2}\Big) < 1.37.
\end{align*}
Thus $N(x)/x$ has no asymptote as $x \to \infty$.

It is interesting to note that $\zeta$ is $4\pi i/\log 2$ periodic.
So $\zeta(s)$ has an analytic continuation to $\{s\!: \Re s=1,\,
s\neq1+4n\pi i/\log 2,\: n\in\mathbb{N}\}$. One could show a larger
oscillation by a more elaborate analysis exploiting additional
singularities.
\medskip

Finally, we discuss $m(x)=\int_{1^{-}}^{x}u^{-1} \mathrm{d}M(u)$ for
this example.  Since $M(x)=\Omega(x)$, we necessarily have
$m(x)=\Omega(1)$. We will prove that
\begin{equation}
\label{m O-bound}
m(x)=O(1).
\end{equation}
This shows that, in general, having Chebyshev bounds, log-density
(\ref{log-density}), $N(x)\ll x $, and (\ref{m O-bound}) together do
not suffice to deduce the estimate $M(x)=o(x)$.

To prove (\ref{m O-bound}), we first need to improve (\ref{log-densityex1}) to  
\begin{equation}
\label{2log-densityex1}
\int_{1^{-}}^{x} \frac{\mathrm{d}N(u)}{u}= \frac{\log x}{\log 2}
-\frac{\log \log x}{2}+O(1). 
\end{equation}
 This estimate can be shown by applying a Tauberian theorem of
 Ingham-Fatou-Riesz type \cite{debruyne-vindasCT}
 (cf.~\cite{ingham1935,korevaarbook}). In fact, the Laplace transform
 of the non-decreasing function $\tau_{1}(x)=\int_{1^{-}}^{e^{x}}
 u^{-1} \mathrm{d}N(u)$ is analytic on $\Re s>0$ and 
\begin{align*}
G(s)&=\mathcal{L}\{\tau_{1}; s\}- \frac{1}{s^{2}\log 2} 
+\frac{\log (1/s)}{2s}=\frac{\zeta(s+1)}{s}-\frac{1}{s^{2}\log 2} 
+\frac{\log (1/s)}{2s} \\ 
&
= \frac{1+\log\log2}{2s}+O(\log^{2} |1/s|), \quad |s|<1/2,
\end{align*}
as a small computation shows. In the terminology of
\cite{debruyne-vindasCT}, $G(s)$ has local pseudo\-measure boundary
behavior on the imaginary segment $(-i/2,i/2)$; in fact, its boundary
value on that segment is the sum of a pseudomeasure and a locally
integrable function.  One then deduces (\ref{2log-densityex1})
directly from \cite[Thm. 3.7]{debruyne-vindasCT}.  
Similarly, we use the fact that $(s\,\zeta(s+1))^{-1}$ has a continuous
extension to the same imaginary segment and we apply the same
Tauberian result to the non-decreasing function
$\tau_{2}(x)=\int_{1^{-}}^{e^{x}} u^{-1}
(\mathrm{d}M(u)+\mathrm{d}N(u))$, whose Laplace transform is   
$\mathcal{L}\{\tau_{2}; s\}= \mathcal{L}\{\tau_{1}; s\}+(s\zeta(s+1))^{-1}.$
The conclusion is again the asymptotic formula $\tau_{2}(\log x)= \log
x/\log 2-(\log \log x)/2+O(1)$.  One then obtains (\ref{m O-bound}) upon
subtracting  (\ref{2log-densityex1}) from this formula. 
\end{example}

\begin{example}
\label{ex2C-M} As a second example, we consider a modification of the
Beurling-Diamond examples from \cite{beurling1937,diamond1970} (see
also \cite{d-s-v}), namely, the continuous prime measure
$\mathrm{d}\Pi_{B}$ given by 
$$
\Pi_{B}(x)= \int_{1}^{x}\frac{1-\cos (\log u)}{\log u}\: \mathrm{d}u  
$$
and the discrete g-prime system 
$$q_{k}=\Pi^{-1}_{B}(k), \quad k=1,2,\dots,$$
with g-prime and g-integer counting functions $\pi_{D}(x)$ and $N_{D}(x)$.

We study here the continuous prime measure
$\mathrm{d}\Pi_{C}=2\mathrm{d}\Pi_{B}$ and the discrete $g$-primes
formed by taking each $q_{k}$ twice, that is, the g-prime system  
$$\mathcal{P}= \{q_{1},\, q_{1},\, q_{2},\, q_{2},\, q_{3},\,
q_{3},\,\dots\}.$$ 
The associated number-theoretic functions will be denoted as
$\pi_{\mathcal{P}}(x)$, $\Pi_{\mathcal{P}}(x)$, $N_{\mathcal{P}}(x)$,
$M_{\mathcal{P}}(x)$, and $\zeta_{\mathcal{P}}(s)$, and those
corresponding to $\mathrm{d}\Pi_{C}$ we denote by $N_{C}(x)$, $M_{C}(x)$, and
$\zeta_{C}(s)$.

It is easy to verify that 
\begin{equation*}
\Pi_{C}(x)=  \frac{x}{\log x}\left(2-\sqrt{2}\cos \left(
\log x-\frac{\pi}{4}\right)\right) +O\left(\frac{x}{\log^{2} x}\right)
\end{equation*}
and, since $\pi_{\mathcal{P}}(x)=2\pi_{D}(x)= 2\lfloor\Pi_{B}(x)\rfloor= \Pi_{C}(x)+ O(1)$,
\begin{equation}
\label{asympdiamondexeq1}
\pi_{P}(x)=\frac{x}{\log x}\left(2-\sqrt{2}\cos \left(
\log x-\frac{\pi}{4}\right)\right)  +O\left(\frac{x}{\log^{2}
x}\right) ,
\end{equation}
whence both $\Pi_{C}$ and $\pi_{P}$ satisfy lower and upper Chebyshev
bounds. The zeta function $\zeta_{C}(s)$ can be explicitly computed:
$$
\log\zeta_{C}(s)= -2\log(s-1)+\log (s-1-i)+\log(s-1+i),
$$
and so
$$
\zeta_{C}(s)= \frac{(s-1)^{2}+1}{(s-1)^{2}}= 1+\frac{1}{(s-1)^{2}}\,.
$$

Now $\zeta_{C}(s)$ has right-hand zeros of order 1 located at $1\pm i$,
and so Lemma \ref{lexC-M} implies that $M_{C}(x)=\Omega(x)$. For the
discrete g-number system, we have
$$
\zeta_{\mathcal{P}}(s)=\zeta_{C}(s)\exp(G(s)),
$$
where $G(s)=\int_{1}^{\infty}x^{-s}\mathrm{d}(\Pi_{\mathcal{P}}-\Pi_{C})(x)$ is
analytic on $\Re s>1/2$ because 
$$\Pi_{\mathcal{P}}(x)=\pi_{\mathcal{P}}(x)+O(\sqrt x)=\Pi_{C}(x)+O(\sqrt x).$$
Thus, the same argument yields $M_{\mathcal{P}}(x)=\Omega(x)$.
 
Note that $\zeta_{C}$ is the Mellin transform of the measure
$\mathrm{d}N_{C}= \delta_{1}+\log u\, \mathrm{d}u$, and therefore we
have the exact formula
$$
N_{C}(x)= x\log x- x+2, \quad x\geq 1.
$$
We now show that $N_{\mathcal{P}}$ satisfies a lower bound of a
similar type.   It is well known \cite{d-s-v,diamond1970} that  
$N_{D}(x)= cx +O(x\log^{-3/2}x)$
with $c>0$. Thus $N_D(x) \ge c'x$ for some $c'>0$ and all 
$x\ge 1$ and so
\begin{align*}
N_{\mathcal{P}}(x) &= \int_{1}^x \mathrm{d}N_D*\mathrm{d}N_D
= \int_{1}^x N_D(x/t)\,\mathrm{d}N_D(t) 
\ge \int_{1}^x c'x/t \,\mathrm{d}N_D(t) \\
&\ge c'\!\!\int_{1}^x N_D(x/t)\,\mathrm{d}t
\ge c'^{\, 2}\!\!\int_{1}^x x/t \,\mathrm{d}t = c'^{\, 2}x \log x \ne O(x).
\end{align*}

Applying the Dirichlet hyperbola method, one can actually obtain the
sharper asymptotic estimate
\begin{equation}
\label{eqasympex2}
N_{\mathcal{P}}(x)=c^{2} x\log x+ bx+ O\left(\frac{x}{\log^{1/2}x}\right)
\end{equation}
with certain constants $b, \, c \in\mathbb{R}$. We leave the
verification of (\ref{eqasympex2})  to the reader. 
\end{example} 
 
\begin{example}  \label{ex:3C-M}
There exists a non-decreasing function $N$ on $[1, \, \infty)$ for which
\[
\int_{1}^{\infty} \frac{|N(x)-x|}{x^{\sigma + 1}}\,
\mathrm{d}x \ll (\sigma -1)^{-1/3}\, \tag{\ref{L1cond} bis}.
\]
holds for $1 < \sigma < 2$, but  $\limsup_{x \to \infty} N(x)/x = \infty$.
\end{example}

To see this, set $f(n) = e^{e^n}$ and take
\[
N(x) = \begin{cases}
x, \ x \ge 1, & x \not\in \displaystyle\bigcup_{m=1}^\infty \big[f(m),
\,e^{m/3}f(m)\big] \\
e^{n/3}f(n), & f(n) \le x \le e^{n/3}f(n).
\end{cases}
\]
Clearly, $N\big(e^{e^n}\big)/ e^{e^n} = e^{n/3} \to \infty.$

On the other hand, for each $n$ we have
\begin{align*}
\int_{f(n)}^{\,e^{n/3}f(n)} \frac{|N(x)-x|}{x^{\sigma + 1}}\,\mathrm{d}x 
&<e^{n/3}f(n)\int_{f(n)}^{\,e^{n/3}f(n)} \frac{\mathrm{d}x}{x^{\sigma + 1}} \\
&< e^{n/3}f(n)\int_{f(n)}^\infty \frac{\mathrm{d}x}{x^{\sigma + 1}}
< e^{n/3} f(n)^{-(\sigma -1)}.
\end{align*}
Instead of summing $ e^{n/3} f(n)^{-(\sigma -1)}$, we calculate the
corresponding integral:  
\[
\int_0^\infty e^{u/3} e^{-(\sigma -1) e^u} \mathrm{d}u 
= \int_0^\infty v^{1/3} e^{-(\sigma -1) v} \frac{\mathrm{d}v}{v} =\Gamma(1/3)  
(\sigma -1)^{-1/3}.
\]
Also, $\{e^{n/3} f(n)^{-(\sigma -1)}\}$ is a unimodal sequence whose maximal
term is of size at most $(3e(\sigma -1))^{-1/3} \ll (\sigma-1)^{-1/3}$. 
Thus (\ref{L1cond} bis) holds.

\


\begin{thebibliography}{99}  
                      

\bibitem{BD69} P.~T.~Bateman, H.~G.~Diamond,
\textit{Asymptotic distribution of Beurling's generalized prime numbers}, 
in \textit{Studies in number theory,} W.~J.~LeVeque, ed., pp. 152--210,
Mathematical Association of America, 1969. 

\bibitem{bateman-diamond2000} \bysame, \emph{On the oscillation
    theorems of Pringsheim and Landau,} in: \emph{Number theory}, pp. 43--54,
  Trends Math., Birkh\"{a}user, Basel, 2000.

\bibitem{BD04} \bysame, \textit{Analytic number theory. An introductory
    course}, World Scientific, Singapore, 2004. Reprinted, with minor
  changes, in Monographs in Number Theory, Vol. 1, 2009.                                                                                            
\bibitem{beurling1937} A.~Beurling, \textit{Analyse de la loi
    asymptotique de la distribution des nombres premiers
    g\'{e}n\'{e}ralis\'{e}s,} Acta Math. \textbf{68} (1937), 255--291.

\bibitem{d-s-v} G.~Debruyne, J.-C.~Schlage-Puchta, J.~Vindas,
  \emph{Some examples in the theory of Beurling's generalized prime
    numbers,} Acta Arith., to appear (preprint: arXiv:1505.04174). 

\bibitem{d-vPNTequiv2016} G.~Debruyne, J.~ Vindas, \emph{On PNT
    equivalences for Beurling numbers,}  preprint (arXiv:1606.03579).  
    
    %Monatsh. Math., to appear
  %(preprint: arXiv:1606.03579).  

\bibitem{debruyne-vindasCT} \bysame, \emph{Complex
    Tauberian theorems for Laplace transforms with local
    pseudofunction boundary behavior,} J. Anal. Math., to appear
  (preprint: arXiv:1604.05069).    

\bibitem{diamond1970} H.~G.~Diamond, \textit{A set of generalized
    numbers showing Beurling's theorem to be sharp,} Illinois
  J. Math. \textbf{14} (1970), 29--34. 

\bibitem{diamond-zhangbook} H.~G.~Diamond, W.-B.~Zhang, \emph{Beurling
    generalized numbers,}  Mathematical Surveys and Monographs series,
  American Mathematical Society, Providence RI, 2016. 

\bibitem{ingham1935} A.~E.~Ingham, \emph{On Wiener's method in
    Tauberian theorems,} Proc. London Math. Soc. (2) \textbf{38}
  (1935), 458--480.  

\bibitem{ingham1942} \bysame,  \textit{On two conjectures in the
    theory of numbers,}  Amer. J. Math. \textbf{64} (1942), 313--319.

\bibitem{korevaarbook} J.~Korevaar, \textit{Tauberian theory. A
    century of developments}, Grundlehren der Mathema\-tischen
  Wissenschaften, Vol.~329, Springer-Verlag, Berlin, 2004.

\bibitem{landau} E.~Landau, \textit{Handbuch der Lehre von der
    Verteilung der Primzahlen}, Teubner, Leipzig-Berlin, 1909,
  reprinted with an appendix by Paul T. Bateman. Chelsea Publishing
  Co., New York, 1953. 

\bibitem{montgomeryBook1970} H.~L.~Montgomery, \emph{Topics in
    multiplicative number theory,} Lecture Notes in Mathematics,
  Vol. 227, Springer-Verlag, Berlin-New York, 1971. 

\bibitem{MV07} H.~L.~Montgomery, R.~C.~Vaughan,
\textit{Multiplicative number theory. I. Classical theory},
Cambridge Studies in Advanced Mathematics, Vol.~97, University
Press, Cambridge, 2007.

\bibitem{zhang1987} W.-B.~Zhang, \textit{A generalization of
    Hal\'{a}sz's theorem to Beurling's generalized integers and its
    application,} Illinois J. Math. \textbf{31} (1987), 645--664.  

\end{thebibliography}
\end{document}